\documentclass[11pt]{amsart}

\usepackage{a4wide}
\usepackage{amssymb}
\usepackage{mathrsfs}
\usepackage{enumerate}
\usepackage{esint}
\usepackage{titletoc}
\usepackage[colorlinks=true, urlcolor=red, linkcolor=red, citecolor=blue]{hyperref}
\usepackage[nameinlink]{cleveref}

\theoremstyle{plain}
\newtheorem{theorem}{Theorem}[section]
\newtheorem{lemma}[theorem]{Lemma}

\newtheorem{proposition}[theorem]{Proposition}

\theoremstyle{definition}
\newtheorem{definition}[theorem]{Definition}
\newtheorem{remark}[theorem]{Remark}

\numberwithin{equation}{section}

\DeclareMathOperator{\supp}{supp}

\DeclareMathOperator{\DIV}{div}

\DeclareMathOperator{\dist}{dist}

\makeatletter
\@namedef{subjclassname@2020}{\textup{2020} Mathematics Subject Classification}
\makeatother

\title[Superharmonic functions and locally renormalized solutions]{Equivalence between Superharmonic functions and renormalized solutions for the  equations with $(p,q)$-growth}
\author[Y. Li, C. Zhang]{Ying Li, Chao Zhang$^*$}
\address{Ying Li\newline
School of Mathematics, Harbin Institute of Technology, Harbin 150001, China\newline
\texttt{lymath@hit.edu.cn}}
\address{Chao Zhang\newline
 School of Mathematics and Institute for Advanced Study in Mathematics, Harbin Institute of Technology, Harbin 150001, China
\newline
\texttt{czhangmath@hit.edu.cn}}
\thanks{$^*$ Corresponding author.}

\begin{document}
\begin{abstract}
We establish the equivalence between superharmonic functions and locally renormalized solutions for the elliptic measure data problems with \((p, q)\)-growth. By showing that locally renormalized solutions are essentially bounded below and using Wolff potential estimates, we extend the results of [T. Kilpel\"{a}inen, T. Kuusi, A. Tuhola-Kujanp\"{a}\"{a}, Superharmonic functions are locally renormalized solutions, Ann. Inst. H. Poincar\'{e} C Anal. Non Lin\'{e}aire, 2011]  to a broader class of problems. Our work provides the first equivalence result between locally renormalized solutions and superharmonic functions for the nonstandard growth equations.
\end{abstract}
\maketitle
\section{Introduction}
Let $\Omega$ be a bounded open domain in $\mathbb{R}^n$. We consider the following problem
\begin{equation}\label{eq:main}
-\DIV\left(|\nabla u|^{p-2}\nabla u+|\nabla u|^{q-2}\nabla u\right)=\mu\quad \text{in}~\Omega,
\end{equation}
where $\mu$ is a nonnegative Radon measure and $1<p\leq q\leq n$. The study of equations with growth between two power-type functions originated from the pioneering work of Marcellini~\cite{ Mar89, Mar91JDE, Mar96, Mar96re}. Since then, such equations have been investigated from a variety of perspectives in the literature~\cite{Bogelein13ARMA, Carozza11poincare, Colombo15ARMA2, colombo15arma, CMM-25-JDE, Schemm11Esaim, Schmidt08Cvpde, Schmidt09ARMA}.
In this paper, we focus on the connection between the superharmonic functions and renormalized solutions for this type of equation when the right-hand side is given by nonnegative measure data.

The concept of renormalized solutions was first introduced by DiPerna and Lions in their seminal work~\cite{DiPerna89on} on the Boltzmann equation. Subsequently, this notion has been extended to the study of some nonlinear elliptic and parabolic problems~\cite{Blanchard97, Boccardo93, Dal99As}.  In particular, Dal Maso et al.~\cite{Dal99As} investigated the existence of renormalized solutions for $p$-growth elliptic equations, where the right-hand side is a bounded measure defined on $\Omega$. For further results on the existence of renormalized solutions for elliptic and parabolic equations with measure data, we refer to ~\cite{Betta02jmpa, Blanchard13MA, Petitta08AMPA, Petitta11jee}.
We mention that Chlebicka~\cite{Chlebicka23mea} established the existence of renormalized solutions for elliptic equations with measure data and generalized Orlicz growth, and also proved that the renormalized solutions and approximable solutions are equivalent. Besides, there are also numerous works devoted to the study of renormalized solutions in the case that $\mu$ belongs to $L^1$, see~\cite{Chlebicka19JDE, Gwiazda18JDE, Gwiazda12JDE, Li24DCDS}.

It is worth pointing out that in the case of nonnegative measure data, Kilpel\"{a}inen and Mal\'{y} \cite{Kil92Degen} established a clear connection between the existence theory and nonlinear potential theory. It was shown that every nonnegative measure induces a superharmonic function for the $p$-growth equation with $p>1$ and the obtained solutions are SOLA as well. Recently, Chlebicka et al.~\cite{CGZ-24} further extended the above results, establishing the existence of superharmonic functions for measure data problems involving operators with Orlicz growth and providing a Wolff potential estimate for these superharmonic functions. The Wolff potential estimate plays a crucial role in potential theory and has become an active area of research in recent years~\cite{Benyaiche23pa, Chlebicka23cvpde, kuusi16non, Kuusi15CMP}. Specifically, Kilpel\"{a}inen et al.~\cite{Kil11super}, based on the potential theory, established the equivalence between superharmonic functions and renormalized solutions for $p$-growth elliptic equations with a non-negative measure as the right-hand side term.  The results mentioned above motivate us to investigate the connection between superharmonic functions and renormalized solutions for more general types of equations.

This paper aims to prove that for the elliptic equation with $(p, q)$-growth and nonnegative measure data, the local renormalized solutions have a superharmonic representation, and further establishes that the superharmonic functions are also renormalized solutions, thereby establishing the equivalence between the local renormalized solutions and superharmonic functions. To prove that a local renormalized solution $u$ has a superharmonic representation, it is essential to show that $u$ is bounded below. Inspired by the literature \cite{Kil94the,kim25wolff}, we obtain the desired estimate by selecting appropriate test functions. The method we developed can be applied to more general equations, provided that the integrability of the renormalized solution is sufficiently regular. Although our approach closely follows that of \cite{Kil11super}, it must address the challenges posed by the increased generality of the growth. We remark that this is the first equivalence result between superharmonic functions and renormalized solutions for the equations with nonstandard growth. 

\smallskip

Our main results are as follows.

\begin{theorem}\label{th:renor}
	Suppose that $u$ is a local renormalized solution to \eqref{eq:main} in $\Omega$ with a  nonnegative Radon measure $\mu$. Then there is a $q$-superharmonic function $\tilde{u}$ such that $\tilde{u}=u$ almost everywhere.
\end{theorem}

\begin{theorem}\label{th-re}
	Let $u$ be a $q$-superharmonic function with the Riesz measure satisfying
	\begin{equation*}
		-\DIV(|\nabla u|^{p-2}\nabla u+|\nabla u|^{q-2}\nabla u)=\mu.
	\end{equation*}
	Then $u$ is a local renormalized solution, i.e.,
	\begin{equation*}
		\int_{\Omega}(|\nabla u|^{p-2}\nabla u+|\nabla u|^{q-2}\nabla u)\cdot\nabla (h(u)\phi)\,dx =\int_{\Omega}h(u)\phi\,d\mu
	\end{equation*}
	for all $\phi\in C^{\infty}_{0}(\Omega)$ and for all Lipschitz functions $h:\mathbb{R}\to \mathbb{R}$ whose derivatives  $h'$ are compactly supported.
\end{theorem}

The paper is organized as follows. Section \ref{sec2} begins with a discussion of some fundamental properties of function spaces and the notion of solutions. In Section  \ref{sec3}, we demonstrate that every local renormalized solution possesses a $q$-superharmonic representative. Finally, Section  \ref{sec4}  establishes that $q$-superharmonic functions are indeed locally renormalized solutions.


\section{Preliminary lemmas} 

\label{sec2}
In this section,  we give the definition of $q$-superharmonic functions and some auxiliary results that will be used later.

\subsection{ $q$-superharmonic functions}

A continuous function
$h\in W_{\rm loc}^{1,q}(\Omega)$ is said to be $q$-harmonic in $\Omega$ if it is a weak solution to
\begin{equation*}
    -\DIV\left( |\nabla h|^{p-2}\nabla h+|\nabla h|^{q-2}\nabla h
    \right)=0,
\end{equation*}
that is,
\begin{equation*}
    \int_{\Omega}\left(|\nabla h|^{p-2}\nabla h+|\nabla h|^{q-2}\nabla h\right) \cdot \nabla \varphi\,dx =0
\end{equation*}
for all $\varphi\in C^{\infty}_0(\Omega)$.

A lower semicontinuous function $u:\Omega\to \mathbb{R}^n\cup\{\infty\}$ is called $q$-superharmonic if $u\not\equiv\infty$ in each component of $\Omega$, and for each open $U\subseteq \Omega$ and each $h\in C(\overline{U})$ that is $q$-harmonic in $U$, the inequality $u\geq h$ on $\partial U$ implies $u\geq h$ in $U$.

A function $u\in W^{1,q}_{\rm loc}(\Omega)$ is called a $q$-supersolution to
\begin{equation}\label{eq-sec2-new}
    \DIV(|\nabla u|^{p-2}\nabla u+|\nabla u|^{q-2}\nabla u)=0
\end{equation}
if
\[ \int_{\Omega}(|\nabla u|^{p-2}\nabla u +|\nabla u|^{q-2}\nabla u)\cdot\nabla \phi\,dx\geq 0\]
for  $0\leq\phi\in C^{\infty}_0(\Omega)$.

 We have the following connections between $q$-superharmonic functions and $q$-supersolutions.


\begin{lemma}[Lemma 4.6,~\cite{Che22Gen}]
    If $u$ is a $q$-superharmonic function in $\Omega$ and locally bounded from above, then $u\in W_{\rm loc}^{1,q}(\Omega)$ and $u$ is $q$-supersolution in $\Omega$.
\end{lemma}

In addition, it has been proved in \cite[Theorem 3.1]{Fang22acv} that
 a distributional supersolution
 to \eqref{eq-sec2-new} is $q$-superharmonic. Thus, we can immediately deduce the following proposition.
\begin{proposition}
    Suppose that $u$ is a finite almost everywhere  function in $\Omega$. Then $u$ has a $q$-superharmonic representative if and only if the truncation $u_k=\min\{u,k\}$ are supersolutions to
    \begin{equation*}
        -\DIV\left(|\nabla u|^{p-2}\nabla u+|\nabla u|^{q-2}\nabla u\right)=0
    \end{equation*}
    for each $k>0$, i.e., $u_k\in W^{1,q}_{\rm loc}(\Omega)$ and
    \begin{equation*}
        \int_{\Omega}\left(|\nabla u_k|^{p-2}\nabla u_k+|\nabla u_k|^{q-2}\nabla u_k\right) \cdot \nabla \varphi \,dx\geq 0
    \end{equation*}
    for all nonnegative $\varphi \in C^{\infty}_0(\Omega)$.
\end{proposition}

An important property of the $q$-superharmonic function is the local summability. See \cite[Theorem 1.1]{Eddaoudi23JEPE} for the proof of the following proposition.

\begin{proposition}\label{prop-sec2-inte}
    If $u$ is $q$-superharmonic in $\Omega$, then $|u|^{q-1}\in L^m_{\rm loc}(\Omega)$ and $|Du|^{q-1}\in L^s_{\rm loc}(\Omega)$ whenever
    \begin{equation*}
        0<m<\frac{n}{n-q}\quad\text{and}\quad0<s<\frac{n}{n-1}.
    \end{equation*}
\end{proposition}

A function $u$ is a solution to
\begin{equation}\label{eq:Sec2}
-\DIV\left(|\nabla u|^{p-2}\nabla u+|\nabla u|^{q-2} \nabla u\right)=\mu,
\end{equation}
if
\begin{equation*}
        \int_{\Omega}\left(|\nabla u|^{p-2}\nabla u+|\nabla u|^{q-2}\nabla u\right)\cdot \nabla \varphi\,dx=\int_{\Omega}\varphi\,d\mu,
\end{equation*}
for all $\varphi\in C^{\infty}_0(\Omega)$.

Indeed, for any nonnegative, bounded Radon measure $\mu$ on $\Omega$, there is a $q$-superharmonic function solving~\eqref{eq:Sec2}, see~\cite[Theorem 1]{CGZ-24}. Conversely, for a finite almost everywhere  $q$-superharmonic function there exists a nonnegative Radon measure $\mu$ such that $u$ solves \eqref{eq:Sec2}. This measure $\mu$ is called the Riesz measure of $u$, and it is often denoted by $\mu[u]$, see~\cite[Proposition 3.14]{CGZ-24}.

\begin{lemma}[Theorem 2, \cite{CGZ-24}]
    If $u$ is a nonnegative, $q$-superharmonic function in $B(x,r)\subset B(x, 2r)\subset \Omega$ for $r\in (0,1)$, which is finite almost everywhere, and if $\mu$ is generated by $u$, then we have
    \begin{equation}\label{eq-sec2-wolff}
        C\left(W_{\mu,r}(x)-r\right)\leq u(x)\leq C\left(\inf_{B(x,r)}u(x)+W_{\mu,r}(x)+r\right),
    \end{equation}
    where
    \begin{equation*}
       W_{\mu,r}(x)=\int^{r}_0g^{-1}\left(\frac{\mu(B(x,s))}{s^{n-1}}\right) \,ds
    \end{equation*}
    and the function $g$ is defined as \begin{equation*}
    g(t)=t^{p-1}+t^{q-1}, \quad \forall t\in\mathbb{R}^+.
\end{equation*}
\end{lemma}
Define a class of functions, namely
\begin{equation*}
    \mathcal{S}_{\mu,r,L}(\Omega')=\left\{u:c_1W_{\mu,r}(x)-c_1r\leq u(x) \leq L+c_2W_{\mu,r}(x), \quad \forall x\in \Omega'\right\}
\end{equation*}
for some  contants $c_1$, $c_2$, $r\in(0,1)$, $L\geq 0$ and $\Omega'\subset\Omega$. Then we have the following lemma.
\begin{lemma}
   Suppose that $u$ is finite almost everywhere and is a nonnegative $q$-superharmonic function with the Riesz measure $\mu$ in a bounded domain $\Omega$. Let $\Omega'\subset \Omega$. For every $0<r<\min\{1, \dist(\Omega',\Omega)/2\}$, there is a constant $L<\infty$ for which $u\in \mathcal{S}_{\mu,r,L}(\Omega')$.
    \end{lemma}
\begin{proof}
    The Wolff potential estimate~\eqref{eq-sec2-wolff} immediately yields the first inequality in the definition of $\mathcal{S}_{\mu,r, L}$. To establish the second inequality, it suffices to find an upper bound for $\inf\limits_{B(y,r)}u$ for any  $y\in \Omega'$. This follows directly from Proposition~\ref{prop-sec2-inte}, which guarantees the existence of a constant $\gamma>0$ such that
    \begin{equation*}
        \inf_{B(y,r)} u\leq \left(\fint_{B(y,r)}u^\gamma\,dx\right)^\frac{1}{\gamma}\leq c\left(r^{-n} \int_{\Omega'+B_{r}}u^\gamma\,dx\right)^{\frac{1}{\gamma}}<+\infty
    \end{equation*}
    for all $y\in \Omega'$, then completing the proof.
\end{proof}

\subsection{Decomposition of measures}
The $q$-capacity $\text{cap}_q(B,\Omega)$ of any $B\subset\Omega$ is defined as follows. For a  compact set $K\subset \Omega$,
\begin{equation*}
    \text{cap}_{q}(K,\Omega):=\inf\left\{ \int_{\Omega}|\nabla \varphi|^q\,dx:\varphi\in C^{\infty}_0(\Omega), \varphi\geq 1~\text{on}~ K\right\}.
\end{equation*}
For open subsets $U\subset \Omega$ and general $ E\subset \Omega$, we have
\begin{equation*}
    \text{cap}_{q}(U,\Omega):=\sup\left\{\text{cap}_{q}(K,\Omega): K\subset U~\text{and}~K~\text{is compact in}~U\right\},
\end{equation*}
and
\begin{equation*}
\text{cap}_{q}(E,\Omega):=\inf\left\{\text{cap}_{q}(U,\Omega):
~E\subset U~\text{and}~U~\text{is open in}~\Omega\right\}.
\end{equation*}
This notion of capacity enjoys all the fundamental properties of classical capacities.
We refer to \cite{BHH18,Chlebicka23mea,CK21remov,HJ22} for more details.

Moreover, for $E\subset \Omega$, we define
\begin{equation*}
\begin{split}
    \widetilde{\text{cap}}_{q}(E,\Omega):=&\sup\bigg\{\nu(E):\nu\in (W^{1,q}_{0}(\Omega))', \supp\nu\subset E, \nu\geq 0,  \\
    &\quad\quad \quad -\text{div}\left(|\nabla \omega|^{p-2}\nabla \omega+|\nabla \omega|^{q-2}\nabla \omega \right)=\nu\quad \text{such that}~0\leq \omega \leq 1\bigg\}.
\end{split}
\end{equation*}
Analogous to the proof in  \cite[Theorem 3.5]{Kil94the}, we have
\begin{equation*}
  \widetilde{\text{cap}}_{q}(E,\Omega)= \text{cap}_{q}(E,\Omega)
\end{equation*}
for $E\subset \Omega$.

A set $E$ is called polar if there exists an open neighborhood $U$ of $E$ and an $q$-superharmonic function $u$ in $U$ such that $u=\infty$ on $E$. A set $E$ is polar if and only if it is of $q$-capacity zero, that is
\begin{equation*}
    \text{cap}_{q}(E\cap U, U)=0
\end{equation*}
for all open sets $U\subset \Omega$.

\begin{remark}
A claim was established in \cite[Theorem 1.5]{HK88} that in the case of $p$-growth, a set $E$ is polar if and only if its $p$-capacity zero. For the case of $(p,q)$-growth, the proof is similar, with the key difference being that in the proof of (3.5) in \cite[Theorem 3.4]{HK88}, the test function $u^{1-p}\phi^p$ is replaced by $(u^{1-p}+u^{1-q})\phi^q$.
\end{remark}

A function $ u$ is said to be $q$-quasicontinuous if, for every $\epsilon > 0$, there exists an open set $U$ with $\text{cap}_q(U) < \epsilon$ such that $u$ is continuous on $\Omega \setminus U$. A property is said to hold $q$-quasieverywhere if it holds everywhere except on a set with $q$-capacity zero.

Let $\mathcal{M}_b$ denote the set of nonnegative bounded measures on $\Omega\subset \mathbb{R}^n$. The subset $\mathcal{M}^1_b$ consists of $q$-diffuse measure, meaning bounded measures   $\mu_0$ that do not charge the sets of $q$-capacity zero. For any bounded nonnegative Radon measure $\mu\in \mathcal{M}_b$, there exists a decomposition $\mu=\mu_{0}+\mu_{s}$, where $\mu_{0}\in \mathcal{M}^1_b$  and $\mu_s$ is singular with respect to the  $q$-capacity, being  concentrated on some set of $q$-capacity zero. Further details on this decomposition can be found in  ~\cite{Chlebicka23mea}. 

In particular, we have the following lemma. The proof follows a similar reasoning as in \cite[Lemma~2.9]{Kil11super}. We present the detailed argument here for completeness.


\begin{lemma}\label{lem-sec2}
    Let $u$ be $q$-superharmonic function that is finite almost everywhere with the Riesz measure $\mu$. Then
    \begin{equation*}
        \mu_s(\{u<+\infty\})=0,
    \end{equation*}
    where $\mu_s$ is the singular part of $\mu$.
\end{lemma}
\begin{proof}
    First, we claim that it was proved in \cite[Theorem 3]{CGZ-24} that if
    \begin{equation*}
        \int_{\Omega}W_{\mu,r}(x)\,d\mu(x)<+\infty
    \end{equation*}
    for a measure $\mu$ and for some $r>0$, then $\mu\in (W^{1,q}_{0}(\Omega))'$, where $(W^{1,q}_{0}(\Omega))'$ is the dual of $W^{1,q}_{0}(\Omega)$.

Let $E\subset\Omega$ be a set such that $\text{cap}_{q}(E)=0$ and $\mu_{s}(\Omega\setminus E)=0$. For every $k>0$ denote $E_{k}=E\cap\{u<k\}$. Fix $k>0$ and let $K\subset E_{k}$ be a compact subset.
Since $K$ is compact, the distance of $K$ and $\partial \Omega$, denoted by $r$, is positive. Using the  Wolff potential estimate \eqref{eq-sec2-wolff}, we obtain
\begin{equation*}
    W_{\mu|K,\frac{r}{8}}(x)\leq W_{\mu,\frac{r}{8}}\leq cu(x)+1<ck
\end{equation*}
for all $x\in K$, where the measure $\mu|K$ is defined by $\mu|K(B)=\mu(K\cap B)$. Thus, we have
\begin{equation*}
    \int_{\Omega}W_{\mu|K,\frac{r}{8}}(x)\,d\mu|K\leq Ck\mu(K)<+\infty,
\end{equation*}
which implies that $\mu|_{K}\in (W^{1,q}_{0}(\Omega))'$.

Let $v\in W^{1,q}_{0}(\Omega)$ be a nonnegative $q$-superharmonic function satisfying
\begin{equation*}
    \DIV(|\nabla v|^{p-2}\nabla v+|\nabla v|^{q-2}\nabla v)=\mu|_{K}
\end{equation*}
in $\Omega$. From the Wolff potential estimate \eqref{eq-sec2-wolff}, we know  that
\[v(x)\leq L+ck\]
for all $x\in K$. Since $v$ is $q$-harmonic in $\Omega\setminus K$, the maximum principle gives
\[0\leq v\leq L+ck\]
in $\Omega$. Let $M=L+ck>1$, and define $\omega=\frac{v}{M}$. Suppose that $\omega$ satisfies the equation
\begin{equation*}
    \DIV(|\nabla \omega|^{p-2}\nabla \omega+|\nabla \omega|^{q-2}\nabla \omega)=\bar{\mu}.
\end{equation*}
 It is easy to check that
\[ M^{1-q}\mu|_{K}\leq \bar{\mu} \leq M^{1-p}\mu|_{K}.\]
Thus, $\bar{\mu}\in (W_0^{1,q}(\Omega))'$.
This shows that $\omega$ is a valid test function for the dual capacity of $K$. We then obtain
\begin{equation*}
   M^{1-q}\mu|_{K}\leq  \widetilde{\text{cap}}_{q}(K,\Omega) = \text{cap}_{q}(K,\Omega)\leq \text{cap}_{q}(E,\Omega)=0.
\end{equation*}
This implies that $\mu(E_k)=0$. Therefore, we have
\begin{equation*}
    \mu_s(\{u<\infty\})\leq \mu_s(\Omega\setminus E)+\sum_{k=1}^{\infty}\mu_s(E_k)=0,
\end{equation*}
which finishes the proof.
\end{proof}

\section{Local Renormalized solution has a $q$-superharmonic representative}

\label{sec3}

In this section, we will prove that every local renormalized solution $u$ admits a $q$-superharmonic representation.  Our methods are heavily influenced by \cite{Kil11super, Kil92Degen, Kil94the,kim25wolff}.

We define the function  $g$  as 
\begin{equation*}
    g(t)=t^{p-1}+t^{q-1},  \quad \forall t\in \mathbb{R}^+.
\end{equation*}
Let $G:[0,\infty)\to [0,\infty)$ be given by
\begin{equation*}
    G(t)=\int^{t}_0 g(\tau)\,d\tau.
\end{equation*}
It follows directly that
\begin{equation}\label{eq:Gg}
    pG(t)\leq tg(t) \leq qG(t) \quad \text{for all} \quad t\geq 0.
\end{equation}
Next, we will present some algebraic inequalities involving $g$ and $G$.
Define $\bar{g}(t)=\frac{G(t)}{t}$. It is straightforward to verify the following inequalities
\begin{equation}
    p\bar{g}(t)\leq g(t)\leq q\bar{g}(t),
\end{equation}
and
\begin{equation}\label{eq:barg}
(p-1)\bar{g}(t)\leq t\bar{g}'(t)\leq (q-1)\bar{g}(t).
\end{equation}

     The following lemma is from \cite[Lemma 2.1]{kim23}.
\begin{lemma}\label{lem:Gg}
Assume that $G$ and $g$ satisfy \eqref{eq:Gg}.
 \begin{enumerate}
 \item  For all $\lambda \geq 1$, it holds that
\begin{equation*}
\lambda^p G(t) \leq G(\lambda t) \leq \lambda^q G(t),\quad \lambda^{p-1} \bar{g}(t) \leq \bar{g}(\lambda t) \leq \lambda^{q-1} \bar{g}(t)
\end{equation*}
and
\[\lambda^{\frac{1}{q-1}}\bar{g}^{-1}(t)\leq \bar{g}^{-1}(\lambda t)\leq \lambda^{\frac{1}{p-1}}\bar{g}^{-1}(t). \]
\item
For all $\lambda \leq 1$, it holds that
\begin{alignat*}{2}
\lambda^q G(t) \leq G(\lambda t) \leq \lambda^p G(t),
\quad\lambda^{q-1} \bar{g}(t) \leq \bar{g}(\lambda t) \leq \lambda^{p-1} \bar{g}(t)
\end{alignat*}
 and
 \[\lambda^{\frac{1}{p-1}}\bar{g}^{-1}(t)\leq \bar{g}^{-1}(\lambda t)\leq \lambda^{\frac{1}{q-1}}\bar{g}^{-1}(t).\]
 \end{enumerate}
\end{lemma}
A further important property of growth functions is given by the following inequality: for any $a, b \geq 0$ and $\varepsilon > 0$, we have the inequality
\begin{equation}\label{eq-alg}
g(a)b \leq \varepsilon g(a)a + g(b/\varepsilon)b.
\end{equation}
This inequality can be derived by splitting the cases
$b \leq \varepsilon a$ and $b > \varepsilon a$.

The \emph{conjugate function} of a Young function $H$ is defined as \[H^\ast(\tau) = \sup_{t \geq 0} (t\tau-H(t))\quad \text{for all}~ \tau \geq 0.\]
We obtain the well-known  \emph{Young's inequality}
\begin{equation}\label{eq-Young}
t\tau \leq H(t) + H^\ast(\tau) \quad\text{for all}~ t, \tau \geq 0,
\end{equation}
by definition. We refer to \cite[Lemma 2.2]{kim23} for the proof of
the following lemma.

\begin{lemma}\label{lem-H}
Let $H(t) = G(t^{1/a})$ for $a \in (0, p]$. Then $H^\ast(G(t)/t^a) \sim G(t)$. In particular, $G^\ast(g(t)) \sim G(t)$.
\end{lemma}

For $k>0$ and for $s\in \mathbb{R}$, we define $T_{k}(s):=\max\{-k,\min\{k,s\}\}$. As shown in \cite[Lemma 2.1]{BPL22}
for every function $u$, such that $T_{k}(u)\in W^{1,q}(\Omega)$ for every $k>0$, there exists a  measurable function $Z_{u}:\Omega\to \mathbb{R}^n$ such that
\[\nabla T_{k}(u)=\chi_{\{|u|<k\}}Z_u\quad \text{for a.e. in }\Omega\text{ and for every } k>0.\]
With an abuse of notation, we denote $Z_u$ simply by $\nabla u$ and call it a generalized gradient.

 Next, we shall give the definition of renormalized solutions.
\begin{definition}\label{Def:renormalized}
  Suppose that $\mu$ is a nonnegative bounded Radon measure in an open set $\Omega$. We say a measurable function $u$ is a local renormalized solution to \eqref{eq:main} in $\Omega$ if
\begin{equation}\label{eq:renormalizedpro}
\begin{split}
    &T_{k}(u) \in W^{1,q}_{\rm loc}(\Omega) \qquad\text{for all}\quad k>0, \\
    &| u|^{q-1}\in L_{\rm loc}^m(\Omega)\quad \text{for all}\quad 1\leq m<\frac{n}{n-q},
\end{split}
\end{equation}
and
\begin{equation}\label{eq-sec3re}
\begin{split}
   &\int_{\Omega}(|\nabla u|^p+|\nabla u|^q)h'(u)\phi\,dx+\int_{\Omega}(|\nabla u|^{p-2}\nabla u+|\nabla u|^{q-2}\nabla u)\cdot \nabla \phi h(u)\,dx\\
   &\quad= \int_{\Omega}h(u)\phi\,d\mu_0+ h(+\infty)\int_{\Omega}\phi\,d\mu_s
\end{split}
\end{equation}
holds for any $h\in W^{1,\infty}(\mathbb{R})$ having $h'$ with compact support and for all $\phi \in C^{\infty}_{0}(\Omega)$.
\end{definition}

As shown in \cite{Chlebicka23mea}, for any nonnegative bounded measure $\mu$, problem \eqref{eq:main} admits a renormalized solution satisfying \eqref{eq-sec3re} with measures such that $\supp\mu_{0}\subset\{|u|<+\infty\}$ and $\supp \mu_s \in \cap_{k>0}\{u>k\}$.  In addition, by \cite[Theorem 1.8]{Chlebicka23mea} and \cite[Remark 1.9]{Chlebicka23mea}, we have the following remark.
\begin{remark}
   For a measurable function $u$ with $T_{k}(u)\in W_0^{1,q}(\Omega)$ is a renormalized solution to problem \eqref{eq:main} if and only if it is an approximable solution. And if $u$ is a renormalized (equivalently approximable) solution, then
   \begin{equation*}
       \int_{\Omega}(|\nabla u|^{p-2}\nabla u+|\nabla u|^{q-2}\nabla u) \cdot \nabla \phi\,dx =\int_{\Omega} \phi\,d\mu
   \end{equation*}
   for all $\phi\in C^{\infty}_0(\Omega)$.
\end{remark}

We proceed to prove that $u$ is locally essentially bounded from below. Let $u$ be a local renormalized solution of \eqref{eq:main} in $\Omega$. Note that $u=u_{+}-u_{-}$. Choose first the test function
    \begin{equation*}
        h_{\epsilon}(u)=\frac{1}{\epsilon}\min\{\epsilon,u_+\}-1,\quad \epsilon>0,
    \end{equation*}
    and let $h\in W^{1,\infty}(\mathbb{R})$ be a non-negative function such that $h'$ has compact support. Let $\phi\in C^{\infty}_{0}(\Omega)$ be nonnegative. On one hand, we have
    \begin{equation*}
        h_{\epsilon}(+\infty)h(+\infty)\int_{\Omega}\phi\,d\mu_s=0, \quad \int_{\Omega}h_{\epsilon}(u)h(u)\phi\,d\mu_0\leq 0,
    \end{equation*}
    and
    \begin{equation*}
        \int_{\Omega}(|\nabla u|^p+|\nabla u|^q)h'_{\epsilon}(u)h(u)\phi\,dx\geq 0.
    \end{equation*}
    On the other hand, the dominated convergence theorem gives
    \begin{equation*}
    \begin{split}
        &\int_{\Omega}\left(|\nabla u|^{p-2}\nabla u+|\nabla u|^{q-2}\nabla u\right)\cdot \nabla (h(u)\phi)h_{\epsilon}(u)\,dx \\
        &\quad \xrightarrow {\epsilon\to 0}\int_{\Omega}\left(|\nabla u|^{p-2}\nabla u_{-}+|\nabla u|^{q-2}\nabla u_{-}\right)\cdot\nabla (h(u)\phi)\,dx.
    \end{split}
    \end{equation*}
    Thus $\bar{u}:=u_{-}$ satisfies
    \begin{equation*}
        \min\{\bar{u},k\}\in W^{1,q}_{\rm loc}(\Omega),\quad k>0,
    \end{equation*}
    and
    \begin{equation}\label{eq:rebound}
        \int_{\Omega}\left(|\nabla \bar{u}|^{p-2}\nabla \bar{u}+|\nabla \bar{u}|^{q-2}\nabla \bar{u}\right)\cdot \nabla (h(\bar{u})\phi)\,dx \leq 0
    \end{equation}
    for all $\phi$ and $h$ as above. This means that $\bar{u}$ is a nonnegative distributional subsolution for which a priori integrability requirements are not necessarily fulfilled.

We now show that $\bar{u}$ is locally bounded which means that $u$ is locally essentially bounded below. Then the assumption $T_{k}(u)\in W_{\rm loc}^{1,q}(\Omega)$ directly implies that $\min\{u,k\}\in W^{1,q}_{\rm loc}(\Omega)$ for all $k>0$.

\begin{lemma}\label{lem:re}
   Assume that $\lambda>0$ and $\delta>1$. Let $u$ be a function defined in $B_R$ and let $\eta$ be a function satisfying $0\leq \eta \leq 1$ and $|\nabla \eta|\leq \frac{C}{R}$. There exist constants $C>0$, depending only on $p$, $q$ and $\delta$, such that
    \begin{equation}\label{eq:leminequ}
  \begin{split}
     &g(|\nabla u|)\frac{\nabla u}{|\nabla u|}\nabla(\Phi(v)\eta^q)\\
     &\quad \geq \frac{C}{R}\left(\frac{\bar{g}(\lambda)}{\bar{g}(\lambda+v)}\right)^{\frac{\delta-1}{q-1}}\frac{G(|\nabla u|)}{\lambda+v}\eta^q-\frac{C\bar{g}(\lambda)}{R}\left(\frac{\bar{g}(\lambda+v)}{\bar{g}(\lambda)}\right)^{\delta},
  \end{split}
    \end{equation}
    where $v=\frac{u}{R}$ and $\Phi(v)=1-\left(\frac{\bar{g}(\lambda)}{\bar{g}(\lambda+v)}\right)^{\frac{\delta-1}{q-1}}$.
\end{lemma}
\begin{proof}
According to \eqref{eq:barg}, we have
    \begin{equation*}
        \begin{split}
            \nabla (\Phi(v)\eta^q)&=\nabla \Phi(v) \eta^q+ q\Phi(v)\eta^{q-1}\nabla \eta\\
            &\geq \frac{(\delta-1)(p-1)}{q-1}\left(\frac{\bar{g}(\lambda)}{\bar{g}(\lambda+v)}\right)^{\frac{\delta-1}{q-1}}\frac{\nabla u}{R(\lambda+v)}\eta^q-q\Phi(v)\eta^{q-1}|\nabla \eta|.
        \end{split}
    \end{equation*}
Let $A$ denote the left-hand side of~\eqref{eq:leminequ}. Then we obtain
\begin{equation*}
    \begin{split}
        A &\geq \frac{(\delta-1)(p-1)p}{(q-1)}\left(\frac{\bar{g}(\lambda)}{\bar{g}(\lambda+v)}\right)^{\frac{\delta-1}{q-1}}\frac{G(|\nabla u|)}{R(\lambda+v)}\eta^q\\
        &\quad -q^2\bar{g}(|\nabla u|)\Phi(v)\eta^{q-1}|\nabla \eta|\\
        &:=I_1+I_2.
    \end{split}
\end{equation*}
By~\eqref{eq-alg}, we have
\begin{equation*}
    \bar{g}(|\nabla u|)\frac{|\nabla \eta|}{\eta}\Phi(v)\leq \epsilon G(|\nabla u|)+\bar{g}\left(\frac{1}{\epsilon}\frac{|\nabla \eta|}{\eta}\right)\frac{|\nabla \eta|}{\eta}
\end{equation*}
for any $\epsilon>0$. Taking
\begin{equation*}
    \epsilon=\frac{(\delta-1)p(p-1)}{2Rq^2(q-1)}\left(\frac{\bar{g}(\lambda)}{\bar{g}(\lambda+v)}\right)^{\frac{\delta-1}{q-1}}\frac{1}{\lambda+v},
\end{equation*}
we deduce that
\begin{equation*}
    I_2 \geq -\frac{1}{2}I_1-q^2\bar{g}\left(\frac{1}{\epsilon}\frac{|\nabla \eta|}{\eta}\right)|\nabla \eta|\eta^{q-1}.
\end{equation*}
Since $\eta \leq 1$, $|\nabla \eta|\leq \frac{C}{R}$ and
 \begin{equation*}
    \epsilon \geq C\left(\frac{\bar{g}(\lambda)}{\bar{g}(\lambda+v)}\right)^{\frac{\delta-1}{q-1}}\frac{1}{R(\lambda+v)},
\end{equation*}
we obtain from Lemma~\ref{lem:Gg} that
\begin{equation*}
    \begin{split}
&\bar{g}\left(\frac{1}{\epsilon}\frac{|\nabla \eta|}{|\eta|}\right)\frac{|\nabla \eta|}{\eta}\eta^{q}\\
&\quad \leq \frac{C}{R}\bar{g}\left(\left(\frac{\bar{g}(\lambda+v)}{\bar{g}(\lambda)}\right)^{\frac{\delta-1}{q-1}}(\lambda+v)\right)\leq \frac{C}{R}\bar{g}(\lambda)\left(\frac{\bar{g}(\lambda+v)}{\bar{g}(\lambda)}\right)^{\delta}.
    \end{split}
\end{equation*}
Thus, we conclude that
\begin{equation*}
    \begin{split}
        A&\geq \frac{C}{R}\left(\frac{\bar{g}(\lambda)}{\bar{g}(\lambda+v)}\right)^{\frac{\delta-1}{q-1}}\frac{G(|\nabla u|)}{\lambda+v}\eta^q
        -\bar{g}(\lambda)\frac{C}{R}\left(\frac{\bar{g}(\lambda+v)}{\bar{g}(\lambda)}\right)^{\delta}.
    \end{split}
\end{equation*}
\end{proof}

\begin{lemma}\label{lem:bound}
   Assume that $1<\delta<\frac{n}{n-(p-1)}$.  Let $u$ be a local renormalized solution to problem \eqref{eq:main} in $B_{R}=B_{R}(x_0)$. Define $v=\frac{\bar{u}}{R}$ with $\bar{u}=u_{-}$. Then, there exists a constant $C_0$ such that
   \begin{equation}\label{eq:b1}
       \left(\fint_{B_{\frac{R}{2}}}\left(\frac{\bar{g}(v)}{\bar{g}(\lambda)}\right)^{\delta}\,dx\right)^{\frac{p}{\gamma}}\leq C_0 \left(\theta^{\frac{p}{\gamma}}_{R}+\fint_{B_R}\left(\frac{\bar{g}(v)}{\bar{g}(\lambda)}\right)^{\delta}\,dx
       \right)
   \end{equation}
   for any $\lambda>0$, where $\gamma$  is a constant to be fixed later. Moreover, there exists a constant $C>0$ such that
   \begin{equation}\label{eq:b2}
       \left(\fint_{B_{\frac{R}{2}}} \bar{g}^{\delta}(v)\,dx\right)^{\frac{1}{\delta}} \leq C\theta^{\rho_1}_{R}\left(\fint_{B_{R}}\bar{g}^{\delta}(v)\,dx\right)^{\frac{1}{\delta}},
   \end{equation}
   where $\rho_1=\frac{1-\frac{p}{\gamma}}{\delta}$ and $\theta_{R}=\frac{|B_{R}\cap \{v>0\}|}{|B_{R}|}$.
\end{lemma}
\begin{proof}
    Since $p<n$, it follows that $p^*=\frac{np}{n-p}<\infty$. Define
    \begin{equation*}
        \tilde{\delta}=\frac{p(q-1)\delta}{q-\delta}.
    \end{equation*}
    It is straightforward to verify that $p<\tilde{\delta}<p^*$.  Choose $\tilde{\delta}<\gamma<p^*$.
    We define a nonnegative function
    \begin{equation*}
        \omega=\left(\frac{\bar{g}(\lambda+v)}{\bar{g}(\lambda)}\right)^{\frac{\delta}{\tilde{\delta}}}-1.
    \end{equation*}
Notice that
\begin{equation}
    \begin{split}
      \left( \frac{\bar{g}(v)}{\bar{g}(\lambda)}\right)^{\delta}&= \left( \frac{\bar{g}(v)}{\bar{g}(\lambda)}\right)^{\delta}\chi_{\{0<v<\lambda\}}+ \left( \frac{\bar{g}(v)}{\bar{g}(\lambda)}\right)^{\delta}\chi_{\{v\geq \lambda\}}\\
      &\leq \chi_{\{0<v\}}+\left( \frac{\bar{g}(v)}{\bar{g}(\lambda)}\right)^{\frac{\delta \gamma}{\tilde{\delta}}}\chi_{\{v\geq \lambda\}}\\
      &\leq \chi_{\{0<v\}}+ (\omega+1)^{\gamma}\chi_{\{v\geq 0\}}\\
      &\leq C\chi_{\{0<v\}}+C\omega^{\gamma},
    \end{split}
\end{equation}
and hence
\begin{equation}\label{eq:gv}
\fint_{B_{\frac{R}{2}}} \left(\frac{\bar{g}(v)}{\bar{g}(\lambda)}\right)^{\delta}\,dx \leq C\theta_{R}+C\fint_{B_{\frac{R}{2}}} \omega^{\gamma}\,dx
\end{equation}
 for some $C>0$.

Let $\eta\in C^{\infty}_0(B_{R})$ be a cut-off function satisfying $0\leq \eta \leq 1$, $\eta \equiv  1$ on $B_{\frac{3R}{4}}$, and $|\nabla \eta|\leq \frac{C}{R}$. Using the Sobolev inequality, we have
\begin{equation}\label{eq:gv1}
    \begin{split}
        \left(\fint_{B_{\frac{R}{2}}} \omega^{\gamma}\,dx\right)^{\frac{p}{\gamma}}&\leq C\left(\fint_{B_{\frac{3R}{4}}} (\omega\eta)^{\gamma}\,dx\right)^{\frac{p}{\gamma}}\\
         &\leq C\fint_{B_{\frac{3R}{4}}}\omega^p\,dx +CR^{p}\fint_{B_{\frac{3R}{4}}} |\nabla \omega|^p\,dx\\
    &=: J_1+J_2.
    \end{split}
\end{equation}

Next, we shall estimate $J_1$ and $J_2$ respectively.

\noindent \textit{Estimate of $J_1$.}
Since $\omega \leq \left(\frac{\bar{g}(2\lambda)}{\bar{g}(\lambda)}\right)^{\frac{\delta}{\tilde{\delta}}}\leq C$ on $\{0<v<\lambda\}$ and $\omega\leq C \left(\frac{\bar{g}(v)}{\bar{g}(\lambda)}\right)^{\frac{\delta}{\tilde{\delta}}}$ on $\{v\geq \lambda\}$, we obtain
\begin{equation}
    J_1 \leq C\theta_{R}+ C\fint_{B_{R}}\left(\frac{\bar{g}(v)}{\bar{g}(\lambda)}\right)^{\delta}\,dx.
\end{equation}

\noindent \textit{Estimate of $J_2$.}  Using \eqref{eq:barg} and Lemma~\ref{lem-H}, we claim that there exists a constant $C>0$ depending only on $p$, $q$ and $\alpha$, such that
\begin{equation}
    \begin{split}
        |\nabla W(v)|^p &\leq \left(\frac{\alpha (q-1)}{\bar{g}^\alpha (\lambda)}\right)^p \frac{\bar{g}^{\alpha p}(\lambda+v)}{(\lambda+v)^p}|\nabla v|^p\\
        &\leq \frac{C}{\bar{g}^{\alpha p}(\lambda)}\frac{\bar{g}^{\alpha p-1}(\lambda+v)}{R^p(\lambda+v)}\left[G(\lambda+v)+G(R|\nabla v|)\right]\\
        &\leq C\frac{\bar{g}^{\alpha p}(\lambda+v)}{R^p\bar{g}^{\alpha p}(\lambda)}+C\frac{G(R|\nabla v|)}{R^p\bar{g}^{\alpha p}(\lambda)}\frac{\bar{g}^{\alpha p}(\lambda+v)}{G(\lambda+v)},
    \end{split}
\end{equation}
where $W(v)=\left(\frac{\bar{g}(\lambda+v)}{\bar{g}(\lambda)}\right)^\alpha-1$ and $\lambda,\alpha>0$. Taking $\alpha=\frac{\delta}{\tilde{\delta}}$, we obtain
\begin{equation}
    \begin{split}
        J_2 &\leq C\fint_{B_{\frac{3R}{4}}}\frac{\bar{g}^{\alpha p}(\lambda+v)}{\bar{g}^{\alpha p}(\lambda)}\,dx+C\fint_{B_{\frac{3R}{4}}}\frac{G(R|\nabla v|)}{\bar{g}^{\alpha p}(\lambda)}\frac{\bar{g}^{\alpha p}(\lambda+v)}{G(\lambda+v)}\,dx\\
        &=:J_{21}+J_{22}.
    \end{split}
\end{equation}

  \textit{Estimate of $J_{21}$.} Similar to the estimates of $J_1$, we have
  \begin{equation}
      J_{21}\leq C\theta_{R}+ C\fint_{B_{R}}\left(\frac{\bar{g}(v)}{\bar{g}(\lambda)}\right)^{\delta}\,dx.
  \end{equation}

\textit{Estimate of $J_{22}$.} Assume that $1<\delta<\frac{n}{n-(p-1)}$. Define an auxiliary function
\begin{equation*}
    h(\bar{u})=\Phi(v)=1-\left(\frac{\bar{g}(\lambda)}{\bar{g}(\lambda+v)}\right)^{\frac{\delta-1}{q-1}},
\end{equation*}
where $v=\frac{\bar{u}}{R}$, and let $\eta\in C^{\infty}_0(B_{R})$ be a cut-off function satisfying $0\leq \eta \leq 1$, $\eta \equiv  1$ on $B_{\frac{3R}{4}}$, and $|\nabla \eta|\leq \frac{C}{R}$. Note that
\begin{equation*}
    h'(\bar{u})=\frac{\delta-1}{q-1}\left(\frac{\bar{g}(\lambda)}{\bar{g}(\lambda+v)}\right)^{\frac{\delta-q}{q-1}}\frac{\bar{g}(\lambda)\bar{g}'(\lambda+v)}{\bar{g}^2(\lambda+v)}>0.
\end{equation*}
Taking $\Phi(v)\eta^q$  as a test function in \eqref{eq:rebound}, we obtain from Lemma~\ref{lem:re} that
\begin{equation*}
\begin{split}
    0&\geq \int_{\Omega}g(|\nabla \bar{u}|)\frac{\nabla \bar{u}}{|\nabla \bar{u}|} \cdot \nabla(\Phi(v)\eta^q)\,dx\\
    & \geq \frac{C}{R}\int_{\Omega}\left(\frac{\bar{g}(\lambda)}{\bar{g}(\lambda+v)}\right)^{\frac{\delta-1}{q-1}}\frac{G(|\nabla \bar{u}|)}{\lambda+v}\eta^q\,dx-\frac{C}{R}\int_{\Omega}\bar{g}(\lambda)\left(\frac{\bar{g}(\lambda+v)}{\bar{g}(\lambda)}\right)^{\delta}\,dx.
\end{split}
\end{equation*}
Noting that $\frac{\delta p}{\tilde{\delta}} -1= \frac{1-\delta}{q-1}$, we deduce from above that
\begin{equation*}
    \begin{split}    J_{22}&=C\fint_{B_{\frac{3R}{4}}}\frac{G(|\nabla \bar{u}|)}{\bar{g}^{\frac{\delta p}{\tilde{\delta}} }(\lambda)}\frac{\bar{g}^{\frac{\delta p}{\tilde{\delta}} }(\lambda+v)}{G(\lambda+v)}\,dx\\
        &\leq C \fint_{B_{R}}\frac{G(|\nabla \bar{u}|)}{\bar{g}^{\frac{\delta p}{\tilde{\delta}} }(\lambda)}\frac{\bar{g}^{\frac{\delta p}{\tilde{\delta}} }(\lambda+v)}{G(\lambda+v)}\eta^q\,dx \\
        &\leq C\fint_{B_R}\left(\frac{\bar{g}(\lambda+v)}{\bar{g}(\lambda)}\right)^{\delta}\,dx\leq C\theta_{R}+ C\fint_{B_{R}}\left(\frac{\bar{g}(v)}{\bar{g}(\lambda)}\right)^{\delta}\,dx.
    \end{split}
\end{equation*}

In addition, it is obvious that $\theta_R\leq \theta^{\frac{p}{\gamma}}_{R}$.
Therefore, we conclude from \eqref{eq:gv} and \eqref{eq:gv1} that
\begin{equation*}
   \left( \fint_{B_{\frac{R}{2}}} \left(\frac{\bar{g}(v)}{\bar{g}(\lambda)}\right)^{\delta}\,dx\right)^{\frac{p}{\gamma}}\leq  C_0\theta^{\frac{p}{\gamma}}_{R}+ C_0\fint_{B_{R}}\left(\frac{\bar{g}(v)}{\bar{g}(\lambda)}\right)^{\delta}\,dx,
\end{equation*}
which gives \eqref{eq:b1}.

Let us next deduce \eqref{eq:b2} from \eqref{eq:b1}. We set
\begin{equation*}
    A_R=\left(\fint_{B_R} \bar{g}^{\delta}(v)\,dx\right)^{\frac{1}{\delta}}
\end{equation*}
and $M=1+C_0$. If we choose $\lambda$ so that
\begin{equation*}
    \frac{1}{\bar{g}(\lambda)}=\frac{M^{\frac{\gamma}{\delta p}}\theta^{\frac{1}{\delta}}_R}{A_{\frac{R}{2}}},
\end{equation*}
then it follows from \eqref{eq:b1} that
\begin{equation*}
    \begin{split}
        \theta^{\frac{p}{\gamma}}_R&=\left(\frac{A_{\frac{R}{2}}}{\bar{g}(\lambda)}\right)^{\frac{\delta p}{\gamma}}-C_0 \theta^{\frac{p}{\gamma}}_{R}\\
        &\leq \left(\fint_{B_{\frac{R}{2}}}\left(\frac{\bar{g}(v)}{\bar{g}(\lambda)}\right)^\delta\,dx\right)^{\frac{p}{\gamma}}-C_0 \theta^{\frac{p}{\gamma}}_{R}\\
        &\leq C_0 \left(\frac{A_R}{\bar{g}(\lambda)}\right)^{\delta}\leq C_0M^{\frac{\gamma}{p}}\theta_R \left(\frac{A_R}{A_{\frac{R}{2}}}\right)^\delta.
    \end{split}
\end{equation*}
This implies that
\begin{equation*}
    A_{\frac{R}{2}} \leq C^{\frac{1}{\delta}}_1\theta^{\rho_1}_{R}A_R,
\end{equation*}
where $\rho_1=\frac{1-\frac{p}{\gamma}}{\delta}>0$. Hence, the desired estimate \eqref{eq:b2} follows.
\end{proof}
\begin{lemma}\label{lem:bound2}
    Let $u$ be a local renormalized solution of \eqref{eq:main} in $\Omega$. Then $u$ is locally essentially bounded below.
\end{lemma}
\begin{proof}
     Let
   $x_0\in\Omega$ be a Lebesgue point of $\bar{g}^{\delta}(u)$ with $1<\delta<\frac{n}{n-(p-1)}$,  and let $R<\dist(x_0,\partial \Omega)$.
 Define $R_j=2^{-j-1}R$ and $B^j=B_{R_j}=B(x_0, R_j)$ for $j=0,1,\cdots$.
For a parameter $\epsilon\in (0,1)$ to be determined later, set $l_0=0$ and
recursively define
 \begin{equation*}
        l_{j+1}=l_j+ R\bar{g}^{-1}\left(\frac{1}{\epsilon}\left(\fint_{B^j}\bar{g}^\delta\left(\frac{(\bar{u}-l_j)_+}{R}\right)\,dx\right)^{\frac{1}{\delta}}\right)
    \end{equation*}
    for $j=0,1,\cdots,$ and $\theta_j=\frac{|B^{j-1}\cap \{u>l_j\}|}{|B^{j-1}|}$.
    Note that we have
    \begin{equation}\label{eq:b3}
        \begin{split}
            \theta_j&\leq \frac{1}{|B^{j-1}|}\int\limits_{B^{j-1}\cap\{u>l_j\}}\left(\frac{\bar{g}((u-l_{j-1})_+/R)}{\bar{g}((l_j-l_{j-1})_+/R)}\right)^\delta\,dx\\
            &\leq \frac{1}{|B^{j-1}|}\int\limits_{B^{j-1}}\left(\frac{\bar{g}((u-l_{j-1})_+/R)}{\bar{g}((l_j-l_{j-1})_+/R)}\right)^\delta\,dx=\epsilon^\delta.
        \end{split}
    \end{equation}
Using Lemma~\ref{lem:bound} to $(\bar{u}-l_j)_+$ in $B^{j-1}$ for $j\geq 1$ and using the inequality above, we obtain
\begin{equation}\label{eq-sec3-barg}
    \bar{g}\left((l_{j+1}-l_j)/R\right)\leq C\epsilon^{\rho_1 \delta}\bar{g}\left((l_j-l_{j-1})/R\right),
\end{equation}
where we used the fact that
\[\bar{g}((u-l_j)_+/R)\leq \bar{g}((u-l_{j-1})_+/R).\]
Taking $\bar{g}^{-1}$ in the both sides of \eqref{eq-sec3-barg}, it yields from Lemma~\ref{lem:Gg} that
\begin{equation*}
    l_{j+1}-l_j \leq C\epsilon^{\frac{\rho_1 \delta}{q-1}}(l_j-l_{j-1}),
\end{equation*}
which implies that  for $k\geq 2$,
\begin{equation*}
    \begin{split}
        l_k&=l_1+\sum^{k-1}\limits_{j=1}(l_{j+1}-l_j)\\
        &\leq l_1 +C\epsilon^{\frac{\rho_1 \delta}{q-1}}l_{k-1}.
    \end{split}
\end{equation*}
Choosing  $\epsilon\in (0,1)$ sufficiently small so that $C\epsilon^{\frac{\rho_1 \delta}{q-1}} <\frac{1}{2}$, we conclude that
\[l_k \leq 2l_1=\bar{g}^{-1}\left(\frac{1}{\epsilon}\left(\fint_{B_{R/2}}\bar{g}^\delta\left(\frac{\bar{u}}{R}\right)\,dx\right)^{\frac{1}{\delta}}\right).\]
It is straightforward to verify that $l_j$ is increasing. Moreover, since $\delta<\frac{n}{n-p+1}<\frac{n}{n-p}$, we deduced from by the assumed summability of $u$  that $l_j$ is  also bounded. Suppose that
$\lim\limits_{j\to \infty}l_j=l$. Therefore, we find
\begin{equation}
\begin{split}
   \bar{g}^\delta\left(\frac{(\bar{u}(x_0)-l)_+}{R}\right)&\leq \lim_{j\to\infty}\fint_{B^j}\bar{g}^\delta\left(\frac{(\bar{u}-l_j)_+}{R}\right)\,dx \\
&=\lim_{j\to\infty}\epsilon^\delta \bar{g}^{\delta}\left(\frac{l_{j+1}-l_j}{R}\right)=0,
\end{split}
\end{equation}
which implies
\begin{equation*}
    u_{-}(x_0)=\bar{u}(x_0)\leq l.
\end{equation*}
Hence, $u$ is locally essentially bounded below.

\end{proof}

Now, we are ready to prove that for nonnegative Radon measure $\mu$ each local renormalized solution has a $q$-superharmonic representative.

\begin{proof}[Proof of Theorem ~\ref{th:renor}.]
    Let $\phi\in C^{\infty}_{0}(\Omega)$ be a non-negative function. For $\epsilon>0$ and $k>0$, define
    \begin{equation}
        h_{k,\epsilon}(t)=\frac{1}{\epsilon}\min\{(k+\epsilon-t)_{+},\epsilon\}.
    \end{equation}
    Since $h'_{k,\epsilon}(t)\leq 0$, it follows that
    \begin{equation*}
        \int_{\Omega}\left(|\nabla u|^{p}+|\nabla u|^q\right)h'_{k,\epsilon}(u)\phi\,dx\leq 0.
    \end{equation*}
    Moreover, the non-negativity of $\mu$ and $\phi$ implies
    \begin{equation*}
        \int_{\Omega}h_{k,\epsilon}(u)\phi\,d\mu_0+h_{k,\epsilon}(+\infty)\int_{\Omega}\phi\,d\mu_s\geq 0.
    \end{equation*}
    Letting $\epsilon\to 0$ and  applying the  dominated convergence theorem, we conclude that
    \begin{equation*}
        \int_{\Omega}\left(|\nabla u_k|^{p}+|\nabla u_k|^q\right)\cdot
    \nabla \phi\,dx \geq 0,
    \end{equation*}
    where $u_k=\min\{u,k\}$.

By Lemma~\ref{lem:bound2}, $u$ is locally bounded from below, which ensures that $u_k\in W_{\rm loc}^{1,q}(\Omega)$ is an ordinary  supersolution.  Consequently, each $u_k$ admits a $q$-superharmonic representative $\tilde{u}_k$. Finally, since $\tilde{u}_k$ forms an increasing sequence of $q$-superharmonic functions, their limit
\[\tilde{u}=\lim\limits_{k\to\infty} \tilde{u}_{k}\]
is also $q$-superharmonic. Thus, $\tilde{u}$ serves as the desired representative of $u$.
\end{proof}

\begin{remark}
     We remark here that the $q$-superharmonic representative
     $\tilde{u}$ obtained from Theorem~\ref{th:renor}  satisfies  that
\begin{equation}\label{eq-sec3-SR}
\begin{split}
\int_{\Omega}h(\tilde{u})\phi\,d\mu&=\int_{\Omega}h(\tilde{u})\phi\,d\mu_0+\int_{\Omega}h(\tilde{u})\phi\,d\mu_s\\
&=\int_{\Omega}h(\tilde{u})\phi\,d\mu_0+h(+\infty)\int_{\Omega}\phi\,d\mu_s
\end{split}
\end{equation}
 for all $\phi\in C^{\infty}_{0}(\Omega)$ and $h\in W^{1,\infty}(\mathbb{R})$ such that $h'$ has a compact support.
Indeed, since $\mu_0$ does not charge sets of $q$-capacity zero,  the integration against $\mu_0$ is independent of the chosen $q$-quasicontinuous representative of $u$. Additionally, by Lemma~\ref{lem-sec2}, we know  that $u=+\infty$ almost everywhere with respect to $\mu_s$. Combining these results, we obtain \eqref{eq-sec3-SR}. Thus, we can write \eqref{eq-sec3re}  for short as
    \begin{equation}\label{eq:Sec3re}
        \int_{\Omega}\left(|\nabla \tilde{u}|^{p-2}\nabla \tilde{u} +|\nabla \tilde{u}|^{q-2}\nabla\tilde{u} \right)\cdot\nabla (h(\tilde{u})\phi)\,dx =\int_{\Omega}h(\tilde{u})\phi\,d\mu.
    \end{equation}
\end{remark}

\section{Superharmonic functions are locally renormalized solutions}
\label{sec4}

\begin{lemma}\label{lem:sec4}
    Let $\mu$ be a nonnegative Radon measure supported in $B(0, R)$. Then there is a $q$-superharmonic function $\omega$ which is finite a.e. solving
   \begin{align}\label{eq:sec4}
   \left\{
  \begin{array}{ccc}
- \DIV~\left(g(|\nabla \omega|)\frac{\nabla \omega}{|\nabla \omega|}\right)= \mu \quad &\mbox{in}& B(0,4R),\\
\omega = 0\quad  &\mbox{on}& \partial B(0,4R),
  \end{array}
  \right.
\end{align}
such that for all $0<r<\min{\{R,1\}}$, there is a positive constant $L<\infty$ such that
\[\omega\in \mathcal{S}_{\mu,r,L}\left(B(0,4R)\right)\]
and
\begin{equation*}
    \int_{B(0,4R)}\left(|\nabla (\min\{\omega,2\lambda\}-\lambda)_+|^p+|\nabla (\min\{\omega,2\lambda\}-\lambda)_+|^q\right)\,dx\leq \lambda \mu\left(\left\{W_{\mu,r}>\frac{\lambda}{L}\right\}\cap B(0,R)\right)
\end{equation*}
for all $\lambda>L$.
\end{lemma}
\begin{proof}
It follows from \cite[Theorem 1.8]{Chlebicka23mea}
that there exists a renormalized solution $v$ to the equation \eqref{eq:sec4} in the formulation vanishing on $\partial B(0,4R)$ in the $W^{1,q}_0$ sense. Moreover, by Theorem~\ref{th:renor}, the renormalized solution has a $q$-superharmonic representative $\omega$ satisfying~\eqref{eq:Sec3re} with $v=\omega$ almost everywhere in  $B(0,4R)$ and $\omega$ is finite almost everywhere in $B(0,4R)$.

Since $\omega$ is a $q$-superharmonic function, it is nonnegative due to the minimum principle. And there exists $L\geq L_0$ such that
\[\omega\in \mathcal{S}_{\mu,r,L}(B(0,2R)).\]
 By the Wolff potential estimate, $\omega$ is locally bounded outside the support of $\mu$, and therefore, $\omega$ is locally in $W^{1,q}$ there. In particular, $\omega$ is $q$-harmonic in $B(0,4R)\setminus B(0,R)$. Using the maximum principle, we obtain
 \[ \sup_{B(0,4R)\setminus B(0,2R)} \omega \leq L_0.\]
Thus, for any $L\geq L_0$, it follows that
\[\omega \in  \mathcal{S}_{\mu,r,L}(B(0,4R)).\]
The potential estimate also gives the inclusion
\[\{\omega>L_0+ck\}\subset \{W_{\mu,r}>k\}\quad \text{for all} ~k>0.\]
Fix $k=\frac{\lambda}{c}-L_0$, where $\lambda>0$, $c>1$. For $\lambda>2cL_0$, we have $k>\frac{\lambda}{2c}$. Therefore, for all $\lambda>2cL_0$, it follows that
\[\{\omega>\lambda\}\subset \left\{W_{\mu,r}>\frac{\lambda}{2c}\right\}.\]

 Next, we study the renormalized equation for  $\omega$ with
 \[h(\omega)=(\min\{\omega,2\lambda\}-\lambda)_+,\quad \text{where} \quad \lambda>2cL_0>0,\]
 which is clearly admissible, as $h$ is Lipschitz continuous and $h'$ has a compact support. Moreover, since $\omega$ vanishes continuously on the boundary of $B(0,4R)$, it follows that $h(\omega)$ has a compact support in $B(0,4R)$ and  $h(\omega)\in W^{1,q}_0(B(0,4R))$. This leads to the following inequality:
 \begin{equation*}
     \begin{split}
         &\lambda\mu\left(\left\{W_{\mu,r}>\frac{\lambda}{2c}\right\}\cap B(0,R)\right)\\
         &\quad \geq \int_{B(0,R)}h(\omega)\,d\mu =\int_{B(0,4R)}g(|\nabla \omega|)\frac{\nabla \omega}{|\nabla \omega|}\cdot\nabla h(\omega)\,dx\\
         &\quad=\int_{\{\lambda<\omega<2\lambda\}\cap B(0,4R)}(|\nabla \omega|^p+|\nabla \omega|^q)\,dx\\
         &\quad =\int_{B(0,4R)}(|\nabla (\min\{\omega,2\lambda\}-\lambda)_+|^p+|\nabla (\min\{\omega,2\lambda\}-\lambda)_+|^q)\,dx,
     \end{split}
 \end{equation*}
and the result follows with $L:=2c\max\{L_0,1\}$.
\end{proof}

\begin{theorem}\label{Th32}
For finite almost everywhere function $u$ and $v$. Suppose that $u$ is a $q$-superharmonic function to \eqref{eq:main}. Suppose further that $v$ is $q$-superharmonic and that for all $\Omega'\subset \Omega$ and for all small $r>0$, there is $L<\infty$ such that
    \[ u,v \in \mathcal{S}_{\mu,r,L}(\Omega').\]
Let $h:\mathbb{R}\times \mathbb{R}\to \mathbb{R}$ be Lipschitz and let $\nabla h$ have a compact support. Then
    \begin{equation*}
        \int_{\Omega}h(u,v)\phi\,d\mu= \int_{\Omega}\left(
        |\nabla u|^{p-2}\nabla u+|\nabla u|^{q-2}\nabla u\right)\cdot\nabla\left(h(u,v)\phi\right)\,dx
    \end{equation*}
    for all $\phi \in C^{\infty}_0(\Omega)$.
\end{theorem}

\begin{proof}
   Define $u_j=\min\{u,j\}$ for $j>0$. Let $k$ be large enough so that $h(u,v)=h(u_k,v_k)$. Let $\phi\in C^{\infty}_0(\Omega)$, and let $\Omega'\subset \Omega$ be a smooth domain such that  $\supp \phi$ belongs to $\Omega'$.

    Let $\epsilon>0$ and
    \begin{equation*}
        K_{\epsilon}\subset \{W_{\mu,1}=+\infty\} \cap \Omega'
    \end{equation*}
    be a compact set satisfying
    \begin{equation*}
        \mu_s(\Omega'\setminus K_\epsilon)<\epsilon.
    \end{equation*}
    Set
    \begin{equation*}
        r=\frac{1}{2}\min\left\{\text{dist}(\Omega',\partial\Omega), \text{dist}(K_\epsilon,\{\max(u,v)\leq k\})\right\}>0
    \end{equation*}
    and define
    \begin{equation*}
        S_{\epsilon}:= \{x\in \mathbb{R}^N:\text{dist}(x,K_{\epsilon})\leq r\}.
    \end{equation*}
    Take $\theta_{\epsilon}\in C_0^{\infty}\left( \{\min\{u,v\}>k\}
        \right)$ with $ 0\leq \theta_\epsilon \leq 1$ and
\begin{equation*}
    \theta_{\epsilon}=1 \quad  \text{on}\quad S_{\epsilon}.
\end{equation*}
 Notice that in the support of $\theta_{\epsilon}\phi$, $h(u,v)=h(k,k)$ is a constant. Define
\begin{equation*}
    \mu_\epsilon =\mu\big|_{\Omega'\setminus K_\epsilon}
\end{equation*}
as the restriction of $\mu$ to the set $\Omega'\setminus K_{\epsilon}$.  Note that for any Borel set $E$,
\[\mu_{\epsilon}(E)\leq \mu_0(E)+\epsilon.\]
It follows that
\begin{equation*}
    W_{\mu,r}=W_{\mu_\epsilon,r} \quad \text{in}\quad \Omega'\setminus S_{\epsilon},
\end{equation*}
which implies that for all $L\geq 0$,
\begin{equation*}
    \mathcal{S}_{\mu_\epsilon,r,L}(\Omega')\subset \mathcal{S}_{\mu,r,L}(\Omega'\setminus S_{\epsilon}).
\end{equation*}

Choose $R$ large enough such that $\Omega'\subset B(0,R)$, and let $\omega_\epsilon$ be a $q$-superharmonic renormalized solution to
\begin{align}\label{1.1}
\begin{cases}
-\DIV\left(|\nabla \omega_\epsilon|^{p-2}\nabla \omega_{\epsilon}+|\nabla \omega_{\epsilon}|^{q-2}\nabla \omega_{\epsilon}\right)=\mu_{\epsilon} \quad &\text{in } B(0,4R),\\
\omega_{\epsilon}=0&\text{on } \partial B(0,4R).
\end{cases}
	\end{align}
By Lemma~\ref{lem:sec4}, there exists a constant $\tilde{L}<\infty$ such that
\begin{equation*}
    \omega_{\epsilon}\in\mathcal{S}_{\mu_{\epsilon},r,\tilde{L}}(\Omega')\subset \mathcal{S}_{\mu,r,\tilde{L}}(\Omega'\setminus S_{\epsilon}),
\end{equation*}
and for
\begin{equation*}
    \psi_{\lambda}=\frac{(\min\{\omega_{\epsilon},2\lambda\}-\lambda)_+}{\lambda},
\end{equation*}
we have the estimate
\begin{equation}\label{eq-sec4-43}
    \begin{split}
&\int_{\Omega}\left(\frac{\nabla (\min\{\omega,2\lambda\}-\lambda)_+|^p}{\lambda^p}+\frac{|\nabla (\min\{\omega,2\lambda\}-\lambda)_+|^q}{\lambda^p}\right)\,dx\\
&\quad \leq C \lambda^{1-p} \mu_{\epsilon}\left(\left\{W_{\mu,r}>\frac{\lambda}{L}\right\}\cap B(0,R)\right)  \\
&\quad \leq C\lambda^{1-p}   \left( \mu_0\left(\left\{W_{\mu,r}>\frac{\lambda}{L}\right\}\cap B(0,R)\right)+\epsilon\right)
\end{split}
\end{equation}
and
\begin{equation}\label{eq-sec4-44}
    \begin{split}
&\int_{\Omega}\left(\frac{\nabla (\min\{\omega,2\lambda\}-\lambda)_+|^p}{\lambda^q}+\frac{|\nabla (\min\{\omega,2\lambda\}-\lambda)_+|^q}{\lambda^q}\right)\,dx\\
&\quad \leq C \lambda^{1-q} \mu_{\epsilon}\left(\left\{W_{\mu,r}>\frac{\lambda}{L}\right\}\cap B(0,R)\right) \\
&\quad \leq C\lambda^{1-q}   \left( \mu_0\left(\left\{W_{\mu,r}>\frac{\lambda}{L}\right\}\cap B(0,R)\right)+\epsilon\right)
\end{split}
\end{equation}
valid for all $\lambda \geq \tilde{L}$.

Moreover, the assumption of the theorem provides $L$ such that
\begin{equation*}
    u,v \in \mathcal{S}_{\mu,r,L}(\Omega'\setminus S_{\epsilon}).
\end{equation*}
Thus, we have
\begin{equation*}
    \omega_\epsilon,u,v\in \mathcal{S}_{\mu,r,\max\{L,\tilde{L}\}}(\Omega'\setminus S_{\epsilon}).
\end{equation*}
This implies that $\omega_\epsilon$, $u$ and $v$ are comparable.  Specifically, there is a constant $C<+\infty$ such that for all $\lambda>C$,
\begin{equation*}
    \{\max\{u,v\}>C\lambda\}\cap(\Omega'\setminus S_{\epsilon})\subset \{\omega_\epsilon>\lambda\}\cap(\Omega'\setminus S_{\epsilon})\subset \{\min\{u,v\}>C^{-1}\lambda\}\cap (\Omega'\setminus S_{\epsilon}).
\end{equation*}

We observe that by the choice of $\theta_\epsilon$, we have
\begin{equation}\label{eq-sec4-3.4}
    \begin{split}
        \int_{\Omega}h(u,v)\phi \theta_{\epsilon}\,d\mu&= h(k,k)\int_{\Omega}h\theta_{\epsilon}\,d\mu\\
        &=h(k,k)\int_{\Omega}\left(|\nabla u|^{p-2}\nabla u+|\nabla u|^{q-2}\nabla u\right)\cdot\nabla (\phi\theta_{\epsilon})\,dx\\
        &=\int_{\Omega}\left(|\nabla u|^{p-2}\nabla u+|\nabla u|^{q-2}\nabla u\right)\cdot\nabla (h(u,v)\phi\theta_{\epsilon} )\,dx.
    \end{split}
\end{equation}
Indeed, $\theta_{\epsilon}$ has been chosen so that its support does not intersect the support of $\nabla h $. Our goal is hence to show that
\begin{equation*}
    \left|\int_{\Omega}h(u,v)\phi(1-\theta_{\epsilon})\,d\mu- \int_{\Omega}(|\nabla u|^{p-2}\nabla u- |\nabla u|^{q-2}\nabla u)\cdot\nabla\left( h(u,v)\phi(1-\theta_{\epsilon})\right)\,dx
    \right|
\end{equation*}
is small by means of $\epsilon$, eventually leading to the result of the theorem.

To demonstrate this,  we employ the truncated equation for $u$, given by
\[-\DIV (|\nabla u_m|^{p-2}\nabla u_m+|\nabla u_m|^{q-2}\nabla u_m)=\mu[u_m],\qquad ~m\in \mathbb{N}.\]

Since both $u_k$ and $v_k$ are $q$-quasicontinuous and belong to $W^{1,q}(\Omega')$, we can find sequence of smooth functions $\{u_{k,j}\}$ and $\{v_{k,j}\}$ that converge to $u_k$ and $v_k$
in the $W^{1,q}(\Omega')$ norm and $q$-quasieverywhere, respectively. Specifically, $u_{k,j}\to u_k$ and $v_{k,j}\to v_{k}$ almost everywhere with respect to $\mu_0$. This immediately leads to the conclusion that $h(u_{k,j},v_{k,j})$ weakly converges to $h(u_k,v_k)$ in $W^{1,q}(\Omega')$. Note that by the choice of $k$, $h(u_k,v_k)=h(u,v)$.

Observe that
\begin{equation}\label{eq-sec4-36}
    \begin{split}
        &\int_{\Omega}h(u_{k,j},v_{k,j})(1-\theta_{\epsilon})\phi \,d\mu[u_m]\\
        &\quad= \int_{\Omega}\psi_{\lambda}h(u_{k,j},v_{k,j})(1-\theta_{\epsilon})\phi\,d\mu[u_m]+ \int_{\Omega}(1-\psi_{\lambda})h(u_{k,j},v_{k,j})(1-\theta_{\epsilon})\phi\,d\mu[u_m].
    \end{split}
\end{equation}

We start by examining the left-hand side of \eqref{eq-sec4-36}. Due to  the weak convergence of $\mu[u_m]$ to $\mu$, we obtain
\begin{equation*}
    \int_{\Omega} h(u_{k,j},v_{k,j})(1-\theta_{\epsilon})\phi\,d\mu[u_m] \to \int_{\Omega} h(u_{k,j},v_{k,j})(1-\theta_{\epsilon})\phi\,d\mu.
\end{equation*}
By the $q$-quasieverywhere convergence and the dominated convergence theorem, we conclude that
\begin{equation*}
    \int_{\Omega}h(u_{k,j},v_{k,j})(1-\theta_{\epsilon})\phi\,d\mu_0 \to \int_{\Omega}h(u,v)(1-\theta_{\epsilon})\phi\,d\mu_0
\end{equation*}
as $j\to \infty$.
Additionally, the estimate
\begin{equation*}
    \int_{\Omega} h(u_{k,j},v_{k,j})(1-\theta_{\epsilon})\phi\,d\mu_s \leq \epsilon \|h\|_{\infty}\|\phi\|_{\infty}
\end{equation*}
holds by Lemma~\ref{lem-sec2} and the choice of $\theta_{\epsilon}$. As a result, we derive the inequality
\begin{equation}
    \limsup_{j\to \infty}\left|\int_{\Omega}h(u_{k,j},v_{k,j})(1-\theta_{\epsilon})\phi\,d\mu- \int_{\Omega}h(u,v)(1-\theta_{\epsilon})\phi\,d\mu
    \right|\leq C\epsilon,
\end{equation}
where $C$ is independent of $\epsilon$. Thus, we have
\begin{equation}\label{eq-sec4-35}
\lim_{j\to\infty}\lim_{m\to\infty}\left|\int_{\Omega} h(u_{k,j},v_{k,j})(1-\theta_{\epsilon})\phi\,d\mu[u_m]-\int_{\Omega}h(u,v)(1-\theta_{\epsilon})\phi\,d\mu\right|\leq \epsilon.
\end{equation}

We now turn our attention to the right-hand side of \eqref{eq-sec4-36}. The first term on the right-hand side is estimated as follows:
\begin{equation}\label{eq-sec4-37}
\int_{\Omega}\psi_{\lambda}h(u_{k,j},v_{k,j})(1-\theta_{\epsilon})\phi\,d\mu[u_{m}] \leq \|h\|_{\infty}\int_{\Omega}\psi_{\lambda}(1-\theta_{\epsilon})\phi\,d\mu[u_m].
\end{equation}
Using the structure of the operator, we obtain the estimate
\begin{equation}\label{eq-sec4-38}
    \begin{split}
        &\int_{\Omega}\psi_{\lambda}(1-\theta_{\epsilon})\phi\,d\mu[u_m]\\
        &\quad=\int_{\Omega}(|\nabla u_m|^{p-2}\nabla u_m+|\nabla u_m|^{q-2}\nabla u_m) \cdot\nabla (\psi_{\lambda}(1-\theta_{\epsilon})\phi)\,dx\\
        &\quad\leq \|\phi\|_{\infty}\int_{\Omega'\setminus S_{\epsilon}}(|\nabla u_m|^{p-1}+|\nabla u_m|^{q-1})\cdot\nabla \psi_{\lambda}\,dx \\
        &\quad \quad + \|\nabla (\phi(1-\theta_{\epsilon}))\|_{\infty}\int_{\Omega'\cap \supp(\psi_\lambda)}(|\nabla u|^{p-1}+|\nabla u|^{q-1})\,dx.
    \end{split}
\end{equation}
Since $\omega_{\epsilon}, u\in \mathcal{S}_{\mu,r,C}(\Omega'\setminus S_{\epsilon})$, we have the bound
\begin{equation}\label{eq-sec4-39}
    u\leq C+cW_{\mu,R}\leq C+c^2\omega_{\epsilon} <C+2c^2\lambda \quad \text{in }  \{\omega_\epsilon<2\lambda\}\cap(\Omega'\setminus S_{\epsilon}),
\end{equation}
which implies that  $u\leq c\lambda$ in the intersection of the support of $\nabla \psi_{\lambda}$ and $\Omega'\setminus S_{\epsilon}$ for all sufficiently large $\lambda$.  In this region, $u_m=u_{c\lambda}$ for all $m>c\lambda$. It follows by H\"{o}lder's inequality  that
\begin{equation}\label{eq-sec4-310}
    \begin{split}
        &\int_{\Omega'\setminus S_{\epsilon}}
        (|\nabla u_m|^{p-2}\nabla u_m+|\nabla u_m|^{q-2}\nabla u_m)\cdot\nabla \psi_{\lambda}\,dx\\
    &\quad \leq \int_{\Omega'\setminus S_{\epsilon}}|\nabla u_{c\lambda}|^{p-1}|\nabla \psi_{\lambda}|\,dx+\int_{\Omega'\setminus S_{\epsilon}}|\nabla u_{c\lambda}|^{q-1}|\nabla \psi_{\lambda}|\,dx\\
    &\quad \leq \left(\int_{\Omega'\setminus S_{\epsilon}}|\nabla u_{c\lambda}|^p\,dx\right)^{\frac{p-1}{p}} \left(\int_{\Omega'}|\nabla \psi_{\lambda}|^p\,dx\right)^{\frac{1}{p}}+\left(\int_{\Omega'\setminus S_{\epsilon}}|\nabla u_{c\lambda}|^q\,dx\right)^{\frac{q-1}{q}}\left(\int_{\Omega'}|\nabla \psi_{\lambda}|^q\,dx\right)^{\frac{1}{q}}.
    \end{split}
\end{equation}
By \eqref{eq-sec4-43} and \eqref{eq-sec4-44}, we obtain
\begin{equation*}
\int_{\Omega'}|\nabla \psi_{\lambda}|^p\,dx \leq C\lambda^{1-p}\left( \mu_0\left(\left\{W_{\mu,r}>\frac{\lambda}{L}\right\}\cap B(0,R)\right)+\epsilon\right)
\end{equation*}
and
\begin{equation*}
\int_{\Omega'}|\nabla \psi_{\lambda}|^q\,dx \leq C\lambda^{1-q}\left( \mu_0\left(\left\{W_{\mu,r}>\frac{\lambda}{L}\right\}\cap B(0,R)\right)+\epsilon\right)
\end{equation*}
for some $\lambda>\tilde{L}$.
Moreover, similar to the proof of~\cite[Theorem 1.13]{Kil92Degen}, we deduce that
\begin{equation}
    \int_{\Omega'\cap\{u\leq \lambda\}}(|\nabla u|^p+|\nabla u|^q)\,dx \leq C\lambda,
\end{equation}
where $C$ is independent of $\lambda$. Thus, it follows from \eqref{eq-sec4-310}  that
\begin{equation}\label{eq-sec4-3100}
\begin{split}
 & \int_{\Omega'\setminus S_{\epsilon}}
        (|\nabla u_m|^{p-2}\nabla u_m+|\nabla u_m|^{q-2}\nabla u_m)\cdot\nabla \psi_{\lambda}\,dx \\
        &\quad \leq C\lambda^{\frac{p-1}{p}}\lambda^{\frac{1-p}{p}}\left(\mu_0\left(\left\{W_{\mu,r}>\frac{\lambda}{L}\right\}\cap \Omega'\right)+\epsilon\right)^{\frac{1}{p}}\\
        &\quad\quad +C\lambda^{\frac{q-1}{q}}\lambda^{\frac{1-q}{q}}\left(\mu_0\left(\left\{W_{\mu,r}>\frac{\lambda}{L}\right\}\cap \Omega'\right)+\epsilon\right)^{\frac{1}{q}} \\
        &\quad \to  C\epsilon^{\frac{1}{p}}+C\epsilon^{\frac{1}{q}}
\end{split}
\end{equation}
as $\lambda\to \infty$ since $\text{cap}_q\left(\{W_{\mu,R}>\frac{\lambda}{C}\}\cap \Omega'\right)\to 0$. Furthermore, the local summability of $|\nabla u|^{q-1}$ implies that
\begin{equation}\label{eq-sec4-312}
    \int_{\Omega'\cap \supp(\psi_{\lambda})}(|\nabla u|^{p-1}+|\nabla u|^{q-1})\,dx \to 0
\end{equation}
as $\lambda\to \infty$ since $\text{cap}_q\left(\{\psi_{\lambda}>0\}\cap \Omega'\right)\to 0$.
By substituting the estimates from~\eqref{eq-sec4-3100} and \eqref{eq-sec4-312} into~\eqref{eq-sec4-38} and then using \eqref{eq-sec4-37}, we arrive at
\begin{equation}
    \lim\limits_{\lambda,j,m\to \infty}\left|\int_{\Omega}h(u_{k,j},v_{k,j})\psi_{\lambda}(1-\theta_{\epsilon})\phi\,d\mu[u_m]\right|\leq C\epsilon.
\end{equation}
Hence by~\eqref{eq-sec4-36} and \eqref{eq-sec4-35},
\begin{equation}\label{eq-sec4-314}
\lim\limits_{\lambda,j,m\to\infty}\left|\int_{\Omega}
h(u_{k,j},v_{k,j})(1-\psi_{\lambda})(1-\theta_{\epsilon})\phi\,d\mu[u_m]-\int_{\Omega}h(u,v)(1-\theta_{\epsilon})\phi\,d\mu\right|\leq C\epsilon,
\end{equation}
where $C$ is independent of $\epsilon$.

Next, we consider the first term on the left-hand side in \eqref{eq-sec4-314}. By~\eqref{eq-sec4-39} we have that $(1-\psi_{\lambda})(1-\theta_{\epsilon})\phi$ vanishes outside $\{u\leq c\lambda\}\cap (\Omega'\setminus S_{\epsilon})$ for all sufficiently large $\lambda$. Hence, for all $m\geq c\lambda$,
\begin{equation*}
    \begin{split}
        &\int_{\Omega} h(u_{k,j},v_{k,j})(1-\psi_{\lambda})(1-\theta_{\epsilon})\phi\,d\mu[u_m]\\
        &\quad=\int_{\Omega'\setminus S_{\epsilon}}g(|\nabla u_{c\lambda}|)\frac{\nabla u_{c\lambda}}{|\nabla u_{c\lambda}|}\cdot\nabla \phi h(u_{k,j},v_{k,j})(1-\theta_{\epsilon})(1-\psi_{\lambda})\,dx\\
        &\quad\quad+\int_{\Omega'\setminus S_{\epsilon}}g(|\nabla u_{c\lambda}|)\frac{\nabla u_{c\lambda}}{|\nabla u_{c\lambda}|}\cdot\nabla(1-\theta_{\epsilon})(1-\psi_{\lambda})\phi h(u_{k,j},v_{k,j})\,dx\\
        &\quad\quad+\int_{\Omega'\setminus S_{\epsilon}}g(|\nabla u_{c\lambda}|)\frac{\nabla u_{c\lambda}}{|\nabla u_{c\lambda}|}\cdot\nabla h(u_{k,j},v_{k,j}) (1-\theta_{\epsilon})(1-\psi_{\lambda})\phi \,dx\\
       &\quad\quad-\int_{\Omega'\setminus S_{\epsilon}}g(|\nabla u_{c\lambda}|)\frac{\nabla u_{c\lambda}}{|\nabla u_{c\lambda}|}\cdot\nabla \psi_{\lambda}
       h(u_{k,j},v_{k,j}) (1-\theta_{\epsilon})\phi \,dx.
    \end{split}
\end{equation*}

We now take the limits of the above expression with respect to $m$, $j$, and $\lambda$ in sequence, starting with $m\to\infty$, followed by $j\to \infty$ and finally $\lambda\to \infty$.

First, it followed  by the dominated convergence theorem that
\begin{equation*}
    \begin{split}
        &\lim_{\lambda\to\infty}\lim_{j\to \infty}\lim_{m\to\infty}\int_{\Omega'\setminus S_{\epsilon}}g(|\nabla u_{c\lambda}|)\frac{\nabla u_{c\lambda}}{|\nabla u_{c\lambda}|}\cdot \nabla \phi h(u_{k,j},v_{k,j})(1-\theta_{\epsilon})(1-\psi_{\lambda})\,dx\\
        &\quad=\int_{\Omega}g(|\nabla u|)\frac{\nabla u}{|\nabla u|}\cdot\nabla \phi h(u,v)(1-\theta_{\epsilon})\,dx
    \end{split}
\end{equation*}
and
\begin{equation*}
    \begin{split}
        &\lim_{\lambda\to\infty}\lim_{j\to \infty}\lim_{m\to\infty}\int_{\Omega'\setminus S_{\epsilon}}g(|\nabla u_{c\lambda}|)\frac{\nabla u_{c\lambda}}{|\nabla u_{c\lambda}|}\cdot\nabla (1-\theta_{\epsilon}) \phi h(u_{k,j},v_{k,j})(1-\psi_{\lambda})\,dx\\
        &\quad=\int_{\Omega}g(|\nabla u|)\frac{\nabla u}{|\nabla u|}\cdot\nabla (1-\theta_{\epsilon}) \phi h(u,v)\,dx.
    \end{split}
\end{equation*}

Second, the weak convergence of $\nabla h(u_{k,j},v_{k,j})$ to $\nabla h(u_k,v_k)=h(u,v)$ together with the dominated convergence gives
\begin{equation*}
    \begin{split}
        &\lim_{\lambda\to\infty}\lim_{j\to \infty}\lim_{m\to\infty}\int_{\Omega'\setminus S_{\epsilon}}g(|\nabla u_{c\lambda}|)\frac{\nabla u_{c\lambda}}{|\nabla u_{c\lambda}|}\cdot\nabla h(u_{k,j},v_{k,j})  (1-\theta_{\epsilon}) \phi (1-\psi_{\lambda})\,dx\\
        &\quad=\int_{\Omega}g(|\nabla u|)\frac{\nabla u}{|\nabla u|}\cdot\nabla h(u,v) (1-\theta_{\epsilon}) \phi \,dx.
    \end{split}
\end{equation*}

Third, estimating as \eqref{eq-sec4-38} and \eqref{eq-sec4-3100}, we have
\begin{equation*}
    \begin{split}
        \left| \int_{\Omega'\setminus S_{\epsilon}}g(|\nabla u_{c\lambda}|)\frac{\nabla u_{c\lambda}}{|\nabla u_{c\lambda}|}\cdot\nabla \psi_{\lambda} h(u_{k,j},v_{k,j})  (1-\theta_{\epsilon}) \phi\,dx
        \right| &\leq C\|h\|_{\infty}\|\phi\|_{\infty}\int_{\Omega'}g(|\nabla u_{c\lambda}|)|\nabla \psi_{\lambda}|\,dx\\
        &\leq C\epsilon,
    \end{split}
\end{equation*}
which readily implies
\begin{equation*}
    \limsup_{\lambda,j,m\to\infty}\left|  \int_{\Omega'\setminus S_{\epsilon}}g(|\nabla u_{c\lambda}|)\frac{\nabla u_{c\lambda}}{|\nabla u_{c\lambda}|}\cdot\nabla \psi_{\lambda} h(u_{k,j},v_{k,j})  (1-\theta_{\epsilon}) \phi\,dx
        \right| \leq C\epsilon.
\end{equation*}

Inserting above estimates into \eqref{eq-sec4-314} we infer that
\begin{equation*}
    \left| \int_{\Omega}h(u,v)(1-\theta_{\epsilon})\phi\,d\mu-\int_{\Omega}g(|\nabla u|)\frac{\nabla u}{|\nabla u|}\cdot\nabla (h(u,v)(1-\theta_{\epsilon})\phi)\,dx\right|\leq C\epsilon.
\end{equation*}
This together with \eqref{eq-sec4-3.4} yields
\begin{equation*}
    \left|\int_{\Omega}h(u,v)\phi\,d\mu-\int_{\Omega}g(|\nabla u|)\frac{\nabla u}{|\nabla u|}\cdot\nabla (h(u,v)\phi)\,dx
    \right|\leq C\epsilon,
\end{equation*}
concluding the proof after letting $\epsilon\to 0$.
\end{proof}

Once Theorem \ref{Th32} is proved, Theorem~\ref{th-re} follows.

 \begin{proof}[Proof of Theorem~\ref{th-re}]
     The proof is completed by choosing $u=v$ in Theorem~\ref{Th32}.
 \end{proof}

\subsection*{Acknowledgment}
 This work was supported by the Postdoctoral Fellowship Program of CPSF under Grant Number: GZC20242215 and the National Natural Science Foundation of China (No. 12471128).

\subsection*{Conflict of interest} The authors declare that there is no conflict of interest. We also declare that this
manuscript has no associated data.

\subsection*{Data availability} Data sharing is not applicable to this article as no datasets were generated or analysed
during the current study.


\end{document}